\documentclass[a4paper, reqno]{amsart}

\usepackage{amsmath, amsthm, amssymb}
\usepackage{color}
\usepackage{fixltx2e}
\usepackage{enumitem}
\usepackage{url}
\usepackage{lmodern}

\theoremstyle{plain}
\newtheorem{thm}{Theorem}[section]
\newtheorem{lemma}[thm]{Lemma}

\numberwithin{equation}{section}

\let\originalleft\left
\let\originalright\right
\renewcommand{\left}{\mathopen{}\mathclose\bgroup\originalleft}
\renewcommand{\right}{\aftergroup\egroup\originalright}
\renewcommand\Re{\operatorname{Re}}
\renewcommand\Im{\operatorname{Im}}

\newcommand\twoln[2]{{\substack{#1 \\ #2}}}

\newcommand\e{e}
\newcommand\emod[2]{e\left( \frac{#1}{#2} \right)}

\newcommand\smod[1]{\, (#1)}
\newcommand\arcosh{\operatorname{arcosh}}

\newcommand\BigO[1]{\mathcal{O}\left(#1\right)}
\newcommand\supp{\operatorname{supp}}

\title{The shifted convolution of divisor functions}
\author{Berke Topacogullari}
\address{Mathematisches Institut, Bunsenstrasse 3-5, D-37073 G\"ottingen, Germany}
\email{btopaco@uni-goettingen.de}

\begin{document}

\bibliographystyle{abbrv}
\subjclass[2010]{Primary 11N37; Secondary 11F30, 11N75}
\keywords{divisor functions, holomorphic cusp forms, shifted convolution sum, Kuznetsov formula}

\begin{abstract}
  We prove an asymptotic formula for the shifted convolution of the divisor functions $d_3(n)$ and $d(n)$, which is uniform in the shift parameter and which has a power-saving error term.
  The method is also applied to give analogous estimates for the shifted convolution of $d_3(n)$ and Fourier coefficents of holomorphic cusp forms.
  These asymptotics improve previous results obtained by several different authors.
\end{abstract}
\maketitle

\section{Introduction}

The binary additive divisor problem is concerned with sums of the form
\[ \sum_{n \leq x} d(n) d(n + h), \quad h \geq 1, \]
where \( d(n) \) is the usual divisor function.
In the past decades a lot of effort has been made to study this problem and several results have been obtained (see \cite{Moto_ThBinAddDivProb} for a historical survey).

Here we will go one step further and look at the sums
\[ D^+(x; h) := \sum_{n \leq x} d_3(n) d(n + h) \quad \text{and} \quad D^-(x; h) := \sum_{n \leq x} d_3(n + h) d(n), \quad h \geq 1, \]
where \( d_3(n) \) is the ternary divisor function.
This problem has also been studied by several authors, beginning with Hooley \cite{Hoo_AsympFormTheNumb}.
The first result with a power-saving error term seems to be given by Deshouillers \cite{Desh_MajMoySomKl}, who used spectral methods to attack a smoothed version of this problem, much in the spirit of his earlier joint work with Iwaniec \cite{DeshIwa_AddDivProb} on the binary additive divisor problem.
Naturally, Deshouillers' result can also be used to treat sums like \( D^{\pm}(x, h) \) with sharp cut-off, although he did not work out the details.

As Friedlander and Iwaniec \cite{FriedIwan_IncomplKlSumsDivProb} pointed out, another approach was possible as a consequence on their work on the ternary divisor function in arithmetic progressions.
Heath-Brown \cite{HB_DivFuncArithProg} improved their result, and showed that
\begin{align}\label{eqn: result of HB}
  D^-(x; 1) = x P(\log x) + \mathcal{O}\left( x^{\frac{101}{102} + \varepsilon} \right)
\end{align}
for any \( \varepsilon > 0 \), where \(P\) is a polynomial of degree three.

Bykovski{\u\i} and Vinogradov \cite{BykVino_InhomConv} returned to the spectral approach of Deshouillers \cite{Desh_MajMoySomKl} based on the Kuznetsov formula and stated \eqref{eqn: result of HB} with an exponent \( \frac89 \) in the error term.
Unfortunately, not more than a few brief hints were given to support this claim, and our first result is a detailed proof of the following asymptotic formula, which yields in addition a substantial range of uniformity in the shift parameter \(h\).

\begin{thm}\label{thm: main theorem for d(n)}
  We have for \( h \ll x^\frac23 \),
  \[ D^\pm(x; h) = x P_h(\log x) + \BigO{ x^{ \frac89 + \varepsilon } }, \]
  where \( P_h \) is a polynomial of degree three, and where the implied constants depend only on \(\varepsilon\).
\end{thm}

Let us also state the analogous result for the smoothed sum.
For a smooth function \( w : \mathbb{R} \rightarrow \mathbb{R} \), which is compactly supported in \( \left[ \frac12, 1 \right] \), define
\[ D_w^\pm(x; h) := \sum_n w\left( \frac nx \right) d_3(n) d(n \pm h). \]
Then we have the following

\begin{thm}\label{thm: main theorem for smoothed d(n)}
  We have for \( h \ll x^\frac23 \),
  \[ D_w^\pm(x; h) = x P_{w, h}(\log x) + \BigO{ x^{\frac56 + \frac\theta3 + \varepsilon} }, \]
  where \( P_{w, h} \) is a polynomial of degree three, and where the implied constants depend at most on \(w\) and \(\varepsilon\).
\end{thm}

By \(\theta\) we denote the bound in the Ramanujan-Petersson conjecture (see section \ref{subsection: The Kuznetsov trace formula and the Large sieve inequalities} for a precise definition).
With the currently best value for \(\theta\) we get an error term which is \( \ll x^\frac{7}{8} \), thus improving the result of Deshouillers \cite{Desh_MajMoySomKl}.

Our method applies as well to the dual sum
\[ D(N) := \sum_{n = 1}^{N - 1} d_3(n) d(N - n). \]
In contrast to the analogous sum with two binary divisor functions (see \cite[Theorem 2]{Moto_ThBinAddDivProb}), the main term in our case is a little bit more complicated.
Our result is

\begin{thm} \label{thm: main theorem for the dual sum of d(n)}
  We have for any \( \varepsilon > 0 \),
  \[ D(N) = M(N) + \BigO{ N^{\frac{11}{12} + \varepsilon} }, \]
  where the main term \( M(N) \) has the form
  \[ M(N) = N \sum_\twoln{ 0 \leq i, j, k, \ell \leq 3 }{ i + j + k + \ell \leq 3 } c_{i, j, k, \ell} F^{ (i, j, k, \ell) } (0, 0, 0, 0) \]
  with certain constants \( c_{i, j, k, \ell} \) and
  \[ F(\alpha, \beta, \gamma, \delta) := N^\alpha \sum_{d \mid N} \frac{ \chi_1(d) }{ d^{1 - \beta} } \sum_{c \mid d} \chi_2(c) \chi_3\left( \frac dc \right), \]
  where the arithmetic functions \( \chi_1 \), \( \chi_2 \) and \( \chi_3 \) are defined by
  \begin{align}
    \begin{gathered} \label{eqn: definition of chi_1, chi_2 and chi_3}
      \chi_1(n) := \prod_{p \mid n} \left( 1 - \frac1{ p^{3 - \gamma - \beta } - p^{ 1 - \gamma + \delta } - p^{1 - \gamma} + 1 } \right), \\
      \chi_2(n) := \prod_{p \mid n} \left( 1 + \frac1{ p^{2 - \beta - \delta} - p^{-\delta} - 1 } \right), \quad \chi_3(n) := \prod_{p \mid n} \left( 1 - \frac1{ p^{1 - \gamma - \delta} } \right).
    \end{gathered}
  \end{align}
  The implied constant depends only on \(\varepsilon\).
\end{thm}

In particular, we have as leading term
\[ D(N) = \left( 1 +  o(1) \right) C_0 C(N) N \log^3 N, \]
where the constant is given by
\[ C_0 := \frac3{\pi^2} \prod_p \left( 1 - \frac1{ p(p + 1) } \right), \]
and where \( C(N) \) is a multiplicative function defined on prime powers by
\[ C\left( p^k \right) := 1 + \left( 1 - \frac1{ p^k } \right) \frac{ 2 p^2 + 2p - 1}{ p^3 - 2p + 1 } - \frac k{p^k} \frac{p + 1}{ (p^2 + p - 1) }. \]

Let \( \varphi(z) \) be a holomorphic cusp form of weight \( \kappa \) for the modular group \( \operatorname{SL}_2( \mathbb{Z} ) \).
Let \( a(n) \) be its normalized Fourier coefficients, so that \( \varphi(z) \) has the Fourier expansion
\[ \varphi(z) = \sum_{n = 1}^\infty a(n) n^\frac{ \kappa - 1 }2 e(nz). \]
The divisor function and the Fourier coefficients \( a(n) \) share a lot of similarities in their behaviour, so one might expect to get analogous results as in Theorems \ref{thm: main theorem for d(n)} and \ref{thm: main theorem for smoothed d(n)} for the sums
\[ A^+(x; h) := \sum_{n \leq x} d_3(n) a(n + h) \quad \text{and} \quad A^-(x; h) := \sum_{n \leq x} d_3(n + h) a(n), \quad h \geq 1, \]
and
\[ A_w^\pm(x; h) := \sum_n w\left( \frac nx \right) d_3(n) a(n \pm h), \]
with the difference that now we cannot expect a main term to appear anymore.
Indeed, Pitt \cite{Pitt_ShiftConvZetaAutLFunc} and Munshi \cite{Mun_ShiftConvDivFuncRamTauFunc} already obtained results of this sort.
Using our method we will be able to partially improve their results by showing

\begin{thm}\label{thm: main theorem for a(n)}
  We have for \( h \ll x^\frac23 \),
  \[ A^\pm(x; h) \ll x^{ \frac89 + \varepsilon } \quad \text{and} \quad A_w^\pm(x; h) \ll x^{\frac56 + \frac\theta3 + \varepsilon}, \]
  where the implied constants depend at most on \(w\), on the holomorphic cusp form \( \varphi(z) \) and on \(\varepsilon\).
\end{thm}

Of course the dual sum
\[ A(N) := \sum_{n = 1}^{N - 1} d_3(n) a(N - n), \]
can be treated as well.

\begin{thm} \label{thm: main theorem for the dual sum of a(n)}
  We have
  \[ A(N) \ll N^{\frac{11}{12} + \varepsilon}, \]
  where the implied constant depends only on \(\varepsilon\).
\end{thm}

As in \cite{BykVino_InhomConv} and \cite{Desh_MajMoySomKl}, our main ingredient is the Kuznetsov trace formula, which enables us to exploit the cancellation between Kloosterman sums.
This approach yields much better error terms than by using results from algebraic geometry to bound complicated exponential sums individually, as it is done in the other works \cite{FriedIwan_IncomplKlSumsDivProb}, \cite{HB_DivFuncArithProg}, \cite{Mun_ShiftConvDivFuncRamTauFunc} and \cite{Pitt_ShiftConvZetaAutLFunc} on \( D^\pm(x; h) \) and \( A^\pm(x; h) \), which give power-saving error terms.

\section{Prerequisites}

Note that \( \varepsilon \) always stands for some positive real number, which can be chosen arbitrarily small.
However, it need not be the same on every occurrence, even if it appears in the same equation.
To avoid confusion we also want to recall that as usually \( e(q) := e^{2\pi i q} \), and that
\[ S(a, b; c) := \sum_\twoln{d \smod c}{ (d, c) = 1 } \emod{ ad + b \overline{d} }c \quad \text{and} \quad c_q(n) := \sum_\twoln{d \smod q}{ (d, q) = 1 } \emod{dn}q, \]
which are the usual notations for Kloosterman sums and Ramanujan sums.
For Kloosterman sums we have Weil's bound,
\[ |S(a, b; c)| \leq d(c) (a, b, c)^\frac12 c^\frac12, \]
while for Ramanujan sums it is well-known that
\[ |c_q(n)| \leq (n, q). \]

\subsection{The Voronoi summation formula and Bessel functions}

Using the well-known Voronoi formula for the divisor function (see \cite[Chapter 4.5]{IwaKow_ANT} or \cite[Theorem 1.6]{Jut_LectMethThExpSums}) and the identity
\[ \sum_\twoln{n = 1}{ n \equiv b \smod c }^\infty d(n) f(n) = \frac1c \sum_{d \mid c} \sum_\twoln{\ell \smod d}{ (\ell, d) = 1 } \emod{-b\ell}d \sum_{n = 1}^\infty d(n) f(n) \emod{n\ell}d, \]
it is not hard to show the following summation formula for the divisor function in arithmetic progressions:

\begin{thm}\label{thm: Voronoi summation for d(n) in arithmetic progressions}
  Let \(b\) and \( c \geq 1 \) be integers.
  Let \( f : (0, \infty) \rightarrow \mathbb{R} \) be smooth and compactly supported.
  Then
  \begin{align*}
    \sum_\twoln{n = 1}{ n \equiv b \smod c }^\infty d(n) f(n) &= \frac1c \int \! \lambda_{b, c}(\xi) f(\xi) \, d\xi \\
      &\qquad -\frac{2 \pi}c \sum_{d \mid c} \sum_{n = 1}^\infty d(n) \frac{ S(b, n; d) }d \int \! Y_0\left( \frac{4 \pi}d \sqrt{n \xi} \right) f(\xi) \, d\xi \\
      &\qquad +\frac4c \sum_{d \mid c} \sum_{n = 1}^\infty d(n) \frac{ S(b, -n; d) }d \int \! K_0\left( \frac{4 \pi}d \sqrt{n \xi} \right) f(\xi) \, d\xi,
  \end{align*}
  with
  \[ \lambda_{b, c}(\xi) = \frac{ \varphi_b(c) }c \log\xi + 2\gamma \frac{ \varphi_b(c) }c - 2 \frac{ \varphi_b'(c) }c, \]
  where
  \[ \varphi_b(c) := c \sum_{d \mid c} \frac{ c_d(b) }d, \quad \text{and} \quad \varphi_b'(c) := c \sum_{d \mid c} \frac{ c_d(b) \log d }d. \]
\end{thm}

In the same way, an analogous formula for Fourier coefficients of holomorphic cusp forms can be obtained by using the corresponding Voronoi formula (see \cite[Theorem 1.6]{Jut_LectMethThExpSums}):

\begin{thm}\label{thm: Voronoi summation for a(n) in arithmetic progressions}
  Let \(b\) and \( c \geq 1 \) be integers.
  Let \( f : (0, \infty) \rightarrow \mathbb{R} \) be smooth and compactly supported.
  Then
  \[ \sum_\twoln{n = 1}{ n \equiv b \smod c } a(n) f(n) = (-1)^\frac\kappa2 \frac{2 \pi}c \sum_{d \mid c} \sum_{n = 1}^\infty a(n) \frac{ S(b, n; d) }d \int_0^\infty \! J_{\kappa - 1} \left( 4\pi \frac{ \sqrt{n \xi} }d \right) f(\xi) \, d\xi. \]
\end{thm}

Here we also want to recall the bounds
\[ d(n) \ll n^\varepsilon \quad \text{and} \quad a(n) \ll n^\varepsilon, \]
the latter following from the Ramanujan-Petersson conjecture proven by Deligne.

Concerning the Bessel function appearing in Theorems \ref{thm: Voronoi summation for d(n) in arithmetic progressions} and \ref{thm: Voronoi summation for a(n) in arithmetic progressions}, we want to sum up some well-known facts.
We know that
\begin{align*}
  K_0(\xi) \ll | \log\xi | \quad \text{for \( \xi \ll 1 \),} \quad \text{and} \quad K_0 \ll \frac1{ e^\xi \sqrt{\xi} } \quad \text{for \( \xi \gg 1 \),}
\end{align*}
and that for \( \mu \geq 1 \),
\[ K_0^{ (\mu) }(\xi) \ll \frac1{ \xi^\mu } \quad \text{for \( \xi \ll 1 \),} \quad \text{and} \quad K_0^{ (\mu) } \ll \frac1{ e^\xi \sqrt{\xi} } \quad \text{for \( \xi \gg 1 \).} \]
Regarding the other two Bessel functions, we have for \( \nu \geq 0 \) and \( \xi \ll 1 \),
\[ J_\nu(\xi) \ll \xi^\nu \quad \text{and} \quad J_\nu^{ (\mu) } \ll \xi^{\nu - \mu} \quad \text{for } \mu \geq 0, \]
and for \( \nu \geq 1 \) and \( \xi \ll 1 \),
\[ Y_0(\xi) \ll | \log\xi |, \quad Y_\nu(\xi) \ll \frac1{ \xi^\nu }, \quad \text{and} \quad Y_0^{ (\mu) } \ll \frac1{ \xi^\mu }, \quad Y_\nu^{ (\mu) } \ll \frac1{ \xi^{\nu + \mu} } \quad \text{for } \mu \geq 1. \]
Finally, for \( \nu \geq 0 \) and \( \xi \gg 1 \), it is known that
\[ J_\nu^{ (\mu) }(\xi), Y_\nu^{ (\mu) }(\xi) \ll \frac1{ \sqrt{\xi} } \quad \text{for \( \mu \geq 0 \).} \]

From the recurrence relations
\begin{align} \label{eqn: recurrence relations for Bessel functions}
  \left( \xi^\nu B_\nu(\xi) \right)' = \xi^\nu B_{\nu - 1}(\xi) \quad \text{and} \quad B_{\nu - 1}(\xi) - B_{\nu + 1}(\xi) = 2 B_\nu'(\xi),
\end{align}
which are true for \( B_\nu(\xi) = J_\nu(\xi) \) or \( B_\nu(\xi) = Y_\nu(\xi) \), we get the identity
\begin{align}\label{eqn: consequence of the recurrence relations}
  \int \! B_0\left( \frac{4\pi}c \sqrt{h \xi} \right) f(\xi) \, d\xi = \left( \frac{-2c}{ 4\pi \sqrt{h} } \right)^\nu \int \! \xi^\frac\nu2 B_\nu\left( \frac{4\pi}c \sqrt{h \xi} \right) \frac{\partial^\nu f}{\partial \xi^\nu}(\xi) \, d\xi,
\end{align}
which will be useful later.
These Bessel functions oscillate for large values, and to make use of this behaviour we have the following

\begin{lemma} \label{lemma: Bessel functions as oscillating functions}
  For any \( \nu \geq 0 \) there are smooth functions \( v_J, v_Y: (0, \infty) \rightarrow \mathbb{C} \) such that
  \begin{align}
    J_\nu(\xi) &= 2\Re\left( \e\left( \frac \xi{2\pi} \right) v_J\left( \frac \xi\pi \right) \right), \\
    Y_\nu(\xi) &= 2\Re\left( \e\left( \frac \xi{2\pi} \right) v_Y\left( \frac \xi\pi \right) \right), \label{eqn: representation of Y_0}
  \end{align}
  and such that for any \( \mu \geq 0 \),
  \begin{align}
    v_J^{ (\mu) }, v_Y^{ (\mu) } \ll \frac1{ \xi^{\mu + \frac12} } \quad \text{for } \xi \gg 1. \label{eqn: estimates for v_J and v_Y}
  \end{align}
\end{lemma}

\begin{proof}
  We start with the integral representations
  \[ J_0(\xi) = \frac1\pi \int_0^\infty \! \sin\left( \frac x{2\pi} + \frac{\pi \xi^2}{2x} \right) \, \frac{dx}x \quad \text{and} \quad Y_0(\xi) = -\frac1\pi \int_0^\infty \! \cos\left( \frac x{2\pi} + \frac{\pi \xi^2}{2x} \right) \, \frac{dx}x, \]
  which can be found in \cite[3.871]{GradRyzh_TableIntSerProd}.
  Here we will only look at \( Y_\nu(\xi) \), as the proof for \( J_\nu(\xi) \) is almost identical.
  As in \cite[Lemma 4]{DeshIwa_AddDivProb}, we use a substitution
  \[ y = \frac{ \sqrt{x} }{2\pi} - \frac{\xi}{ 2 \sqrt{x} }, \quad x = \pi^2 \left( y + \sqrt{ y^2 + \frac\xi\pi } \right)^2, \]
  so that we can write the integral above as
  \[ Y_0(\xi) = -\frac2\pi \int_{-\infty}^\infty \! \cos\left( 2\pi \left( y^2 + \frac\xi{2\pi} \right) \right) \left( y^2 + \frac\xi\pi \right)^{-\frac12} \, dy. \]
  Now writing the cosine function out as a sum of exponential functions, we get \eqref{eqn: representation of Y_0} for \( Y_0 \) with
  \[ v_Y(\xi) = -\frac2\pi \int_0^\infty \! \frac{ \e(y^2) }{ \sqrt{ y^2 + \xi } } \, dy. \]
  The estimate \eqref{eqn: estimates for v_J and v_Y} can be shown by splitting the integral at \(1\) and repeatedly using partial integration on the part which goes to \(\infty\).
  The statements for \( Y_\nu(\xi) \) follow from \eqref{eqn: recurrence relations for Bessel functions}.
\end{proof}

\subsection{The Kuznetsov trace formula and the Large sieve inequalities}\label{subsection: The Kuznetsov trace formula and the Large sieve inequalities}

We follow in great parts the notation used in \cite{DeshIwan_KlSumsFourCoeffCuspForms}.
Let \(q\) be some positive integer, which will stay fixed throughout this section, and let \( \Gamma := \Gamma_0(q) \) be the Hecke congruence subgroup of level \(q\).
For these groups we have the spectral decomposition
\[ L^2(\Gamma \backslash \mathbb{H}) = \mathbb{C} \oplus L_\text{cusp}^2(\Gamma \backslash \mathbb{H}) \oplus L_\text{Eis}^2(\Gamma \backslash \mathbb{H}), \]
where \( L_\text{cusp}^2( \Gamma \backslash \mathbb{H} ) \) is the space spanned by the cusp forms, and \( L_\text{Eis}^2( \Gamma \backslash \mathbb{H} ) \) is a continuous sum spanned by Eisenstein series.

Let \( u_0  \) be the constant function, and let \( u_j \), \( j \geq 1 \), run over an orthonormal basis of \( L_\text{cusp}^2(\Gamma \backslash \mathbb{H}) \), with the corresponding real eigenvalues \( \lambda_0 < \lambda_1 \leq \lambda_2 \leq \ldots \).
We set \( {\kappa_j}^2 = \lambda_j - \frac14 \), where we choose the sign of \( \kappa_j \) so that \( i \kappa_j \geq 0 \) if \( \lambda_j < \frac14 \), and \( \kappa_j \geq 0 \) if \( \lambda_j \geq \frac14 \).
Then the Fourier expansions of these functions is given by
\[ u_j(z) = y^\frac12 \sum_{m \neq 0} \rho_j(m) K_{i \kappa_j}(2 \pi |m| y) e(mx). \]
The Selberg eigenvalue conjecture says that \( \lambda_1 \geq \frac14 \), which would imply that all \( \kappa_j \) are real and non-negative, however this still remains to be proven.
The eigenvalues with \( 0 < \lambda_j < \frac14 \) as well as the corresponding values \( \kappa_j \) are called exceptional, and lower bounds for these exceptional \( \lambda_j \) imply upper bounds for the corresponding \( i \kappa_j \).
Let \( \theta \in \mathbb{R}_0^+ \) be such that \( i \kappa_j \leq \theta \) for all exceptional \( \kappa_j \) uniformly for all levels \(q\); by the work of Kim and Sarnak \cite{Kim_FunctExtSqGL4SymFouGL2} we know that we can choose
\begin{align}\label{eqn: Kim-Sarnak-bound}
  \theta = \frac7{64}.
\end{align}

For any cusp \( \mathfrak{c} \) of \( \Gamma \) we have the Eisenstein series, defined for \( \Re s > 1 \) and \( z \in \mathbb{H} \) by
\[ E_\mathfrak{c}(z; s) = \sum_{ \tau \in \Gamma_\mathfrak{c} \backslash \Gamma } \Im( \sigma_\mathfrak{c}^{-1} \tau z )^s, \]
which can be continued meromorphically to the whole complex plane.
The space \( L_\text{Eis}^2( \Gamma \backslash \mathbb{H} ) \) is then the continuous direct sum spanned by the \( E_\mathfrak{c}(z; \frac12 + ir) \), \( r \in \mathbb{R} \), and the Fourier expansion of these Eisenstein series around \( \infty \) is given by
\begin{align*}
  E_\mathfrak{c}(z; s) = \delta_{ \mathfrak{c} \infty } y^s &+ \pi^\frac12 \frac{ \Gamma\left( s - \frac12 \right) }{ \Gamma(s) } \varphi_{ \mathfrak{c}, 0 }(s) y^{1 - s} \\
    &+ 2 y^\frac12 \frac{ \pi^s }{ \Gamma(s) } \sum_{n \neq 0} |n|^{s - \frac12} \varphi_{ \mathfrak{c}, n }(s) K_{s - \frac12}(2 \pi |n| y) e(nx).
\end{align*}

Finally, denote by \( \mathfrak{M}_k(\Gamma) \) the space of holomorphic cusp forms of weight \(k\) and by \( \theta_k(q) \) its dimension.
Let \( f_{j, k} \), \( 1 \leq j \leq \theta_k(q) \), be an orthonormal basis of \( \mathfrak{M}_k(\Gamma) \).
Then the Fourier expansion of \( f_{j, k} \) around \(\infty\) is given by
\[ f_{j, k}(z) = \sum_{m = 1}^\infty \psi_{j, k}(m) e(mz). \]\

With the whole notation set up, we can now formulate the famous Kuznetsov trace formula (see \cite[Theorem 1]{DeshIwan_KlSumsFourCoeffCuspForms}).

\begin{thm}\label{thm: Kuznetsov trace formula}
  Let \( f : (0, \infty) \rightarrow \mathbb{C} \) be smooth with compact support.
  Let \(m\), \(n\) be two positive integers.
  Then
  \begin{align*}
    \sum_{ c \equiv 0 \smod q } \frac{ S(m, n; c) }c &f\left( 4\pi \frac{ \sqrt{mn} }c \right) = \sum_{j = 1}^\infty \frac{ \overline{ \rho_j }(m) \rho_j(n) }{ \cosh(\pi \kappa_j) } \hat f( \kappa_j ) \\
      &+ \frac1\pi \sum_\mathfrak{c} \int_{-\infty}^\infty \! \left( \frac mn \right)^{-ir} \overline{ \varphi_{ \mathfrak{c}, m } } \left( \frac12 + ir \right) \varphi_{\mathfrak{c}, n} \left( \frac12 + ir \right) \hat f(r) \, dr \\
      &+ \frac1{2 \pi} \sum_\twoln{ k \equiv 0 \smod 2 }{ 1 \leq j \leq \theta_k(q) } \frac{ i^k (k - 1)! }{ \left( 4\pi \sqrt{mn} \right)^{k - 1} } \overline{ \psi_{j, k} }(m) \psi_{j, k}(n) \tilde f(k - 1),
  \intertext{and}
    \sum_{ c \equiv 0 \smod q } \frac{ S(m, -n; c) }c &f\left( 4\pi \frac{ \sqrt{mn} }c \right) = \sum_{j = 1}^\infty \frac{ \rho_j(m) \rho_j(n) }{ \cosh(\pi \kappa_j) } \check f( \kappa_j ) \\
      &+ \frac1\pi \sum_\mathfrak{c} \int_{-\infty}^\infty \! (mn)^{ir} \varphi_{ \mathfrak{c}, m } \left( \frac12 + ir \right) \varphi_{\mathfrak{c}, n} \left( \frac12 + ir \right) \check f(r) \, dr,
  \end{align*}
  where the Bessel transforms are defined by
  \begin{align*}
    \hat f(r) &= \frac\pi{ \sinh(\pi r) } \int_0^\infty \! \frac{ J_{2ir}(\xi) - J_{-2ir}(\xi) }{2i} f(\xi) \, \frac{d\xi}\xi, \\
    \check f(r) &= \frac4\pi \cosh(\pi r) \int_0^\infty \! K_{2ir}(\xi) f(\xi) \, \frac{d\xi}\xi, \\
    \tilde f(\ell) &= \int_0^\infty \! J_\ell(\xi) f(\xi) \, \frac{d\xi}\xi.
  \end{align*}
\end{thm}

To get some first estimates for the appearing Bessel transforms we refer to \cite[Lemma 2.1]{BloHarMich_BurgSubconvBoundTwiLFunc}:

\begin{lemma}\label{lemma: estimates for the Bessel transforms}
  Let \( f : (0, \infty) \rightarrow \mathbb{C} \) be a smooth and compactly supported function such that
  \[ \supp f \asymp X \quad \text{and} \quad f^{ (\nu) } \ll \frac1{Y^\nu} \quad \text{for} \quad \nu = 0, 1, 2, \]
  for positive \(X\) and \(Y\) with \( X \gg Y \).
  Then
  \begin{align}
    \hat f(ir), \check f(ir) &\ll \frac{ 1 + Y^{-2r} }{1 + Y} & \text{for} \quad 0 \leq r < \frac14, \label{eqn: first estimate for the bessel transforms} \\
    \hat f(r), \check f(r), \tilde f(r) &\ll \frac{ 1 + |\log Y| }{1 + Y} & \text{for} \quad r \geq 0, \\
    \hat f(r), \check f(r), \tilde f(r) &\ll \left( \frac XY \right)^2 \left( \frac1{ r^\frac52 } + \frac X{ r^3 } \right) & \text{for} \quad r \gg \max(X, 1).
  \end{align}
\end{lemma}

For oscillating functions, we can do better.
Assume \( w : (0, \infty) \rightarrow \mathbb{C} \) to be a smooth and compactly supported function such that
\[ \supp w \asymp X \quad \text{and} \quad w^{ (\nu) } \ll \frac1{X^\nu} \quad \text{for} \quad \nu \geq 0, \]
and for \( \alpha > 0 \) define
\[ f(\xi) := \e\left( \xi \frac\alpha{2 \pi} \right) w(\xi). \]
Then the following two lemmas give bounds for the Bessel transforms of \( f(\xi) \), depending on the sizes of \(X\) and \(\alpha\).

\begin{lemma}\label{lemma: estimates for the Bessel transforms of oscillating functions for large alpha}
  Assume that
  \[ X \ll 1 \quad \text{and} \quad \alpha X \gg 1. \]
  Then for \( \nu, \mu \geq 0 \),
  \begin{align}
    \hat f(ir), \check f(ir) &\ll X^{ -2r + \varepsilon } \left( X^\mu + \frac1{ (\alpha X)^\nu } \right) &\text{for} \quad 0 < r \leq \frac14, \label{eqn: first estimate for the bessel transforms for osciallting functions} \\
    \hat f(r), \check f(r), \tilde f(r) &\ll \frac{\alpha^\varepsilon}{\alpha X} \left( \frac{\alpha X}r \right)^\nu &\text{for} \quad r > 0. \label{eqn: second estimate for the bessel transforms for osciallting functions}
  \end{align}
\end{lemma}

\begin{proof}
  We begin with \eqref{eqn: first estimate for the bessel transforms for osciallting functions}.
  Using the Taylor series of the \( J_\nu \)-Bessel function we can write the Bessel transform \( \hat f(ir) \) as
  \begin{align}
    \hat f(ir) = \frac\pi2 \sum_{m = 0}^\infty \frac{ (-1)^m }{ 4^m m! } \int_0^\infty \! e\left( \xi \frac\alpha{2 \pi} \right) g(\xi, r, m) w(\xi) \xi^{2m - 1} \, d\xi \label{eqn: Taylor series for hat f}
  \end{align}
  with
  \[ g(\xi, r, m) := \frac1{ \sin(\pi r) } \left( \frac1{ \Gamma(m + 2r + 1) } \left( \frac\xi2 \right)^{2r} - \frac1{ \Gamma(m - 2r + 1) } \left( \frac\xi2 \right)^{-2r} \right). \]
  For \( 0 < r \leq \frac12 \), one can check that we have the bound
  \[ \frac{ \partial^\nu g }{ \partial \xi^\nu } (\xi, r, m) \ll \frac{ X^{-2r + \nu + \varepsilon} }{ (m - 1)! }. \]
  By splitting the sum in \eqref{eqn: Taylor series for hat f} at \( m = \mu \), and using partial integration for the finite part while estimating trivially the rest, we get that
  \[ \hat f(ir) \ll X^{-2r + \varepsilon} \left( X^{2\mu} + \frac1{ (\alpha X)^\nu } \right). \]
  The estimate for \( \check f(ir) \) follows in exactly the same way by using the corresponding Taylor series for \( K_{2ir}(\xi) \).
  
  For the proof of \eqref{eqn: second estimate for the bessel transforms for osciallting functions} we follow \cite[Lemma 3]{Jut_ConvFourCoeffCuspForms}.
  We begin with the following identity (see \cite[8.411.11]{GradRyzh_TableIntSerProd}),
  \begin{align} \label{eqn: identity for J-Bessel function}
    \frac{ J_{2ir}(\eta) - J_{-2ir}(\eta) }{ \sinh(\pi r) } = \frac2{\pi i} \int_{-\infty}^\infty \! \cos( \eta \cosh\zeta ) \cos(2c \zeta) \, d\zeta,
  \end{align}
  which gives
  \[ \hat f(r) = - \int_{-\infty}^\infty \int \! \cos(\eta \cosh\zeta) \cos(2r \zeta) f(\eta) \, \frac{d\eta}\eta d\zeta =: -(I^+ + I^-) \]
  with
  \[ I^\pm := \int_{-\infty}^\infty \int \! \e\left( \eta \left( \frac{ \alpha \pm \cosh \zeta }{2\pi} \right) \right) \frac{ w(\eta) }\eta \cos(2r \zeta) \, d\eta d\zeta. \]
  To bound \( I^+ \) we use partial integration \(\mu\)-times on the integral over \(\eta\) and get
  \[ I^+ \ll \frac{ \alpha^\varepsilon }{ (\alpha X)^\mu }. \]
  
  The treatment of \( I^- \) is a little trickier since the factor
  \[ \gamma(\zeta) := \alpha - \cosh \zeta \]
  occuring in the exponent may vanish, so that we have to treat the integral differently depending on whether \( \gamma(\zeta) \) is near \(0\) or not.
  Out of technical reasons, it is easier to use smooth weight functions to split the integral.
  Set
  \[ Z_1 := \arcosh(\alpha - A), \quad Z_2 := \arcosh(\alpha + A), \quad \text{with} \quad A := \frac1X. \]
  Let \( u_i : \mathbb{R} \rightarrow [0, \infty) \), \( i = 1, 2, 3 \), be suitable weight functions such that
  \begin{align*}
    u_1(\xi) &= 1 \quad \text{for} \quad | \xi | \leq \frac12 Z_1 \quad \text{and} \quad \supp u_1 \subseteq [ -Z_1, Z_1 ], \\
    u_2(\xi) &= 1 \quad \text{for} \quad | \xi | \geq 2 Z_2 \quad \text{and} \quad \supp u_2 \subseteq [ -\infty, -Z_2 ] \cup [ Z_2, \infty ],
  \end{align*}
  and define
  \[ u_3(\xi) := 1 - u_1(\xi) - u_2(\xi). \]
  Note that for all \( i = 1, 2, 3 \),
  \[ u_i^{ (\nu) }(\xi) \ll 1 \quad \text{for} \quad \nu \geq 0. \]
  Then we have to consider the integrals
  \begin{align} \label{eqn: splitted I^-}
    I_i^- := \int \int \! u_i(\zeta) \e\left( \eta \frac{ \gamma(\zeta) }{2\pi} \right) \frac{ w(\eta) }\eta \cos(2r \zeta) \, d\eta d\zeta,
  \end{align}
  and using partial integration \(\mu\)-times over \(\eta\) we get
  \[ I_1^-, I_2^- \ll \frac A{ \alpha (XA)^\mu } + \frac{ \alpha^\varepsilon }{ (\alpha X)^\mu }, \]
  whereas bounding \( I_3^- \) directly gives
  \[ I_3^- \ll \frac A\alpha. \]
  This already proves \eqref{eqn: second estimate for the bessel transforms for osciallting functions} for \( \nu = 0 \).
  The result for \( \nu \geq 1 \) can be shown the same way by partially integrating \(\nu\)-times over \(\zeta\) before estimating the integrals absolutely.
  
  The estimate for \( \check f(r) \) can be shown analogously by using the integral representation
  \[ K_{2ir}(\eta) = \frac1{ \cosh(\pi r) } \int_0^\infty \! \cos(\eta \sinh\zeta) \cos(2r \zeta) \, d\zeta \]
  (see \cite[8.432.4]{GradRyzh_TableIntSerProd}).
  Finally, the proof for \( \tilde f(r) \) also goes along the same lines -- in this case we use the identity
  \[ J_\ell(\eta) = \frac1\pi \int_0^\pi \! \cos( \ell \zeta - \eta \sin\zeta ) \, d\zeta, \]
  which can be found for instance in \cite[8.411.1]{GradRyzh_TableIntSerProd}.
\end{proof}

\begin{lemma}\label{lemma: estimates for the Bessel transforms of oscillating functions for alpha near 1}
  Assume that
  \[ X \gg 1 \quad \text{and} \quad | \alpha - 1 | \ll \frac{ X^\varepsilon }X. \]
  Then for \( \nu \geq 0 \),
  \begin{align}
    \hat f(ir), \check f(ir) &\ll 1 &\text{for} \quad 0 < r \leq \frac14, \label{eqn: first estimate for alpha near 1} \\
    \hat f(r), \tilde f(r) &\ll \frac{X^\varepsilon}{ X^\frac12 } \left( \frac{ X^\frac12 }r \right)^\nu &\text{for} \quad r > 0, \label{eqn: second estimate for alpha near 1} \\
    \check f(r) &\ll \frac{X^\varepsilon}X \left( \frac Xr \right)^\nu &\text{for} \quad r > 0.
  \end{align}
\end{lemma}

\begin{proof}
  The first bound \eqref{eqn: first estimate for alpha near 1} follows directly from \eqref{eqn: first estimate for the bessel transforms}.
  The proof of the other bounds follows the same path as in Lemma \ref{lemma: estimates for the Bessel transforms of oscillating functions for large alpha}, so we only want to point out some differences.
  In the case of \( \hat f(r) \), we again use the identity \eqref{eqn: identity for J-Bessel function}.
  For \( I^+ \) we here get the bound
  \[ I^+ \ll \frac1{ X^\mu }. \]
  It is again necessary to split \( I^- \), and in order to do so, we choose a suitable weight function \( u_1(\xi) \) which satisfies
  \[ u_1(\xi) = 1 \quad \text{for} \quad |\xi| \geq 2Z, \quad u_1(\xi) = 0 \quad \text{for} \quad |\xi| \leq Z, \]
  and
  \[ u_1^{ (\nu) }(\xi) \ll \frac1{ Z^\nu } \asymp \frac1{ A^\frac\nu2 }, \]
  where
  \[ A := \frac{ X^\varepsilon }X \quad \text{and} \quad Z := \arcosh(2A + \alpha). \]
  Set \( u_2(\xi) := 1 - u_1(\xi) \).
  Then
  \[ I^- =: I_1^- + I_2^- \]
  in the same way as in \eqref{eqn: splitted I^-}, and we get
  \[ I_1^- \ll \frac{ A^\frac12 }{ (XA)^\mu } + \frac1{ X^\mu } \ll \frac{ X^\varepsilon }{ X^\frac12 } \quad \text{and} \quad I_2^- \ll A^\frac12 \ll \frac{ X^\varepsilon }{ X^\frac12 }. \]
  This gives \eqref{eqn: second estimate for alpha near 1} for \( \nu = 0 \).
  By partially integrating over \(\zeta\), we get the result for higher \(\nu\).
  Finally, the results for \( \tilde f(r) \) and \( \check f(r) \) can be deduced similarly by using the appropriate integral representations for the occuring Bessel functions.
\end{proof}

Another important tool are the large sieve inequalities for Fourier coefficients of cusp forms and Eisenstein series (see \cite[Theorem 2]{DeshIwan_KlSumsFourCoeffCuspForms}).
For a sequence \( a_n \) of complex numbers define
\[ \| a_n \|_N := \sqrt{ \sum_{N < n \leq 2N} |a_n|^2 }, \]
and furthermore set
\begin{align*}
  \Sigma_j^{(1)} (N) &:= \frac1{ \sqrt{ \cosh(\pi \kappa_j) } } \sum_{N < n \leq 2N} a_n \rho_j(n), \\
  \Sigma_{ \mathfrak{c}, r }^{(2)} (N) &:= \sum_{N < n \leq 2N} a_n n^{ir} \varphi_{ \mathfrak{c}, n } \left( \frac12 + ir \right), \\
  \Sigma_{j, k}^{(3)} (N) &:= i^\frac k2 \sqrt{ \frac{ (k - 1)! }{ (4\pi)^{k - 1} } } \sum_{N < n \leq 2N} a_n n^{ -\frac{k - 1}2 } \psi_{j, k}(n).
\end{align*}

Then we have the following

\begin{thm}\label{thm: large sieve inequalities}
  Let \( K \geq 1 \) and \( N \geq \frac12 \) be real numbers, \( a_n \) a sequence of complex numbers and \( \mathfrak{c} \) a cusp of \( \Gamma \).
  Then
  \begin{align*}
    \sum_{ | \kappa_j | \leq K } \left| \Sigma_j^{(1)} (N) \right|^2 &\ll \left( K^2 + \frac{ N^{1 + \varepsilon} }q \right) \| a_n \|_N^2, \\
    \sum_\mathfrak{c} \int_{-K}^K \! \left| \Sigma_{ \mathfrak{c}, r }^{(2)} (N) \right|^2 \, dr &\ll  \left( K^2 + \frac{ N^{1 + \varepsilon} }q \right) \| a_n \|_N^2, \\
    \sum_\twoln{2 \leq k \leq K, \, 2 \mid k}{ 1 \leq j \leq \theta_k(q) } \left| \Sigma_{k, j}^{(3)}(N) \right|^2 &\ll \left( K^2 + \frac{ N^{1 + \varepsilon} }q \right) \| a_n \|_N^2,
  \end{align*}
  where the implicit constants depend only on \( \varepsilon \).
\end{thm}

When there is no averaging over \(n\), the bounds given by the large sieve inequalities are not optimal.
So, we also want to mention

\begin{lemma}\label{lemma: estimates for Fourier coefficients}
  Let \( K \geq 1 \) and \( n \geq 1 \).
  Then
  \begin{align}
    \sum_{ | \kappa_j | \leq K } \frac{ | \rho_j(n) |^2 }{ \cosh(\pi \kappa_j) } &\ll K^2 + (qKn)^\varepsilon (q, n)^\frac12 \frac{ n^\frac12 }q, \label{eqn: first estimate for Fourier coefficients} \\
    \sum_\mathfrak{c} \int_{-K}^K \! \left| \varphi_{ \mathfrak{c}, n } \left( \frac12 + ir \right) \right|^2 dr &\ll K^2 + (qKn)^\varepsilon (q, n)^\frac12 \frac{ n^\frac12 }q, \label{eqn: second estimate for Fourier coefficients} \\
    \sum_\twoln{2 \leq k \leq K, \, 2 \mid k}{ 1 \leq j \leq \theta_k(q) } \frac{ (k - 1)! }{ (4\pi n)^{k - 1} } \left| \psi_{j, k}(n) \right|^2 &\ll K^2 + (qKn)^\varepsilon (q, n)^\frac12 \frac{ n^\frac12 }q, \label{eqn: third estimate for Fourier coefficients}
  \end{align}
  where the implicit constants depend only on \( \varepsilon \).
\end{lemma}

\begin{proof}
  For the full modular group, \eqref{eqn: first estimate for Fourier coefficients} and \eqref{eqn: second estimate for Fourier coefficients} are proven in \cite[Lemma 2.4]{Moto_SpecTheRieZetaFunc}.
  Except for some obvious modifications, the proof applies as well to general Hecke congruence subgroups.
  The proof of \eqref{eqn: third estimate for Fourier coefficients} is a simpler variant of the proof of \cite[Proposition 4]{DeshIwan_KlSumsFourCoeffCuspForms}.
\end{proof}

Finally, to treat the exceptional eigenvalues we need a result, which we state here as

\begin{lemma} \label{lemma: lemma to treat the exceptional eigenvalues}
  Let \( X, q, h \geq 1 \) be such that \( h^\frac12 X \geq q \).
  Then
  \[ \sum_{ \kappa_j \text{ exc.} } | \rho_j(h) |^2 X^{ 4i \kappa_j } \ll (Xh)^\varepsilon \frac{ h^{2\theta} X^{4\theta} }{ q^{4\theta} } (h, q)^\frac12 \left( 1 + \frac{ h^\frac12 }q \right), \]
  where the implicit constant depends only on \( \varepsilon \).
\end{lemma}

\begin{proof}
  We have
  \[ \sum_{ \kappa_j \text{ exc.} } | \rho_j(h) |^2 X^{ 4i \kappa_j } \leq \left( \frac{ h^\frac12 X }q \right)^{4\theta} \sum_{ \kappa_j \text{ exc.} } | \rho_j(h) |^2 \left( 1 + \frac q{ h^\frac12 } \right)^{ 4i \kappa_j }. \]
  To treat the sum on the right hand side we make use of \cite[(16.58)]{IwaKow_ANT}, which says that
  \[ \sum_{ \kappa_j \text{ exc.} } | \rho_j(h) |^2 Y^{ 4i \kappa_j } \ll (qYh)^\varepsilon (h, q)^\frac12 \frac{ h^\frac12 Y }q \quad \text{for} \quad Y \geq 1, \]
  and the result follows.
\end{proof}

\section{Proof of Theorems \ref{thm: main theorem for d(n)} and \ref{thm: main theorem for a(n)}} \label{section: main proof}

Our method applies to \( D^\pm(x; h) \) as well as \( A^\pm(x; h) \), and it will pose no further difficulty to treat both cases simultaneously.
With this in mind, we let \( \alpha(n) \) be a placeholder for \( d(n) \) or \( a(n) \).

From now on we consider \(x\) and \(h\) as fixed.
Let \( w : \mathbb{R} \rightarrow [0, \infty) \) be a smooth function with compact support in \( \left[ \frac12, 1 \right] \) such that
\[ w^{ (\nu) } \ll \frac1{ \Omega^\nu } \quad \text{for } \nu \geq 0, \quad \text{and} \quad \int \! \big| w^{ (\nu) }(\xi) \big| d\xi \ll \frac1{ \Omega^{\nu - 1} } \, d\xi \quad \text{for} \quad \nu \geq 1, \]
where \( \Omega := x^{-\omega} \) with \( 0 \leq \omega < \frac16. \)
We will look at the sum
\begin{align*}
  \Phi(w) := \sum_n d_3(n) \alpha(n + h) w\left( \frac nx \right), \quad h \in \mathbb{Z},  h \neq 0,
\end{align*}
with the aim of showing that
\begin{align}\label{eqn: main sum}
  \Phi(w) = M(w) + \BigO{ x^{\frac56 + \varepsilon} \left( x^\frac\theta3 + x^\frac\omega2 \right) \left( 1 + \frac{ h^\frac14}{ x^\frac16 } \right) }
\end{align}
for \( h \ll x^\frac56 \) and any \( \varepsilon > 0 \); recall that \(\theta\) was defined at \eqref{eqn: Kim-Sarnak-bound}.
The main term \( M(w) \) vanishes if \( \alpha = a \), and otherwise has the form \( M(w) = x Q_{w, h}(\log x) \) with a cubic polynomial \( Q_{w, h} \).
The choice \( \omega = 0 \) gives Theorem \ref{thm: main theorem for smoothed d(n)} and the second bound in Theorem \ref{thm: main theorem for a(n)}, while the choice \( \omega = \frac19 \) together with a suitable weight function \(w\) gives Theorem \ref{thm: main theorem for d(n)} and the first bound in Theorem \ref{thm: main theorem for a(n)}.

We will need a smooth decomposition of the ternary divisor function, for which we will use a similar construction as the one used in \cite{Meur_BinAddDivProb}.
Let \( v : \mathbb{R} \rightarrow [0, \infty) \) be a smooth function such that
\[ \supp v \subset [-2, 2], \quad \text{and} \quad v(\xi) = 1 \quad \text{for } \xi \in [-1, 1], \]
and define
\[ v_1(\xi) := v\left( \frac\xi{ x^\frac13 } \right), \quad v_2(\xi) := v\left( \frac\xi{ \sqrt{ \frac xa } } \right). \]
If \( abc \leq x \), then obviously
\[ \left( v_1(a) - 1 \right) \left( v_1(b) - 1 \right) \left( v_1(c) - 1 \right) = 0 \quad \text{and} \quad \left( v_2(b) - 1 \right) \left( v_2(c) - 1 \right) = 0, \]
and hence
\[ d\left( \frac na \right) = \sum_{bc = \frac na} v_2(b) \left( 2 - v_2(c) \right) \]
as well as
\begin{align}
  d_3(n) &= \sum_{abc = n} \left( v_1(a) v_1(b) v_1(c) - 3 v_1(a) v_1(b) + 3 v_1(a) \right) \notag \\
    &= \sum_{abc = n} \left( v_1(a) v_1(b) v_1(c) - 3 v_1(a) v_1(b) \right) + 3 \sum_{a \mid n} d\left( \frac na \right) \notag \\
    &= \sum_{abc = n} h(a, b, c) \label{eqn: decomposition of d_3(n)}
\end{align}
with
\[ h(a, b, c) := v_1(a) v_1(b) v_1(c) - 3 v_1(a) v_1(b) + 3 v_1(a) v_2(b) ( 2 - v_2(c) ). \]
Note that this function is non-zero only when
\[ a, b \ll c. \]

It will be useful to use a partition of unity on \( (0, \infty) \) constructed as follows.
Let \( h_X \) be smooth and compactly supported functions such that
\[ \supp h_X \subset \left[ \frac X2, 2X \right], \quad h^{ (\nu) } \ll \frac1{ X^\nu } \quad \text{and} \quad \sum_X h_X = 1, \]
where the last sum runs over powers of \(2\).
Then we set
\[ h_{ABC}(a, b, c) := h(a, b, c) h_A(a) h_B(a) h_C(c) \]
and
\[ \Phi_{ABC}(w) := \sum_{a, b, c} h_{ABC}(a, b, c) \alpha(abc + h) w\left( \frac{abc}x \right), \]
so that
\[ \Phi(w) = \sum_{A, B, C} \Phi_{ABC}(w). \]
Note that we can bound the derivatives of \( h_{ABC} \) by
\[ \frac{ \partial^{\nu_1 + \nu_2 + \nu_3} }{ \partial a^{\nu_1} \partial b^{\nu_2} \partial c^{\nu_3} } h_{ABC}(a, b, c) \ll \frac1{ A^{\nu_1} B^{\nu_2} C^{\nu_3} }. \]
Furthermore we can assume
\begin{align*}
  ABC \asymp x \quad \text{and} \quad A \ll B \ll C,
\end{align*}
since otherwise \( \Phi_{ABC}(w) \) is empty, and since our argument is symmetric in \(A\) and \(B\).
This also implies that
\[ A \ll x^\frac13, \quad AB^2 \ll x \quad \text{and} \quad AB \ll x^\frac23. \]

\subsection{Use of the Voronoi summation formula}

We have
\begin{align*}
  \Phi_{ABC}(w) &= \sum_{a, b} \sum_{ m \equiv h \smod {ab} } \alpha(m) w\left( \frac{m - h}x \right) h_{ABC} \left( a, b, \frac{m - h}{ab} \right) \\
    &= \sum_{a, b} \sum_{ m \equiv h \smod{ab} } \alpha(m) f(m; a, b),
\end{align*}
where we have set
\[ f(\xi; a, b) := w\left( \frac{\xi - h}x \right) h_{ABC} \left( a, b, \frac{\xi - h}{ab} \right). \]
Note that
\[ \supp f( \,\bullet\, ; a, b ) \asymp x \quad \text{and} \quad \frac{ \partial^{\nu_1 + \nu_2} }{ \partial \xi^{\nu_1} \partial b^{\nu_2} } f(\xi; a, b) \ll \frac1{ (x \Omega)^{\nu_1} B^{\nu_2} }. \]

Now we use Theorem \ref{thm: Voronoi summation for d(n) in arithmetic progressions} in the case \( \alpha = d \) to get
\begin{align*}
  \Phi_{ABC}(w) &= \sum_{a, b} \frac1{ab} \int \! \lambda_{h, ab}(\xi) f(\xi; a, b) \, d\xi \\
    &\qquad - 2\pi \sum_\twoln{a, b, c}{c \mid ab} \frac1{ab} \sum_m d(m) \frac{ S(h, m; c) }c \int_0^\infty \! Y_0\left( \frac{4\pi}c \sqrt{m \xi} \right) f(\xi; a, b) \, d\xi \\
    &\qquad + 4 \sum_\twoln{a, b, c}{c \mid ab} \frac1{ab} \sum_m d(m) \frac{ S(h, -m; c) }c \int_0^\infty \! K_0\left( \frac{4\pi}c \sqrt{m \xi} \right) f(\xi; a, b) \, d\xi,
\end{align*}
and Theorem \ref{thm: Voronoi summation for a(n) in arithmetic progressions} in the case \( \alpha = a \), which gives
\[ \Phi_{ABC}(w) = (-1)^\frac\kappa2 2\pi \sum_\twoln{a, b, c}{c \mid ab} \frac1{ab} \sum_m a(m) \frac{ S(h, m; c) }c \int_0^\infty \! J_{\kappa - 1} \left( 4\pi \frac{ \sqrt{m\xi} }c \right) f(\xi; a, b) \, d\xi. \]
The possible main term will be given by
\[ M_0(w) := \sum_{A, B, C} \sum_{a, b} \frac1{ab} \int \! \lambda_{h, ab}(\xi) f(\xi; a, b) \, d\xi, \]
which we will compute at the end.
First we want to treat the other sums and show that they are small enough.

Here we can restate the outer sum as follows
\begin{align*}
  \sum_\twoln{a, b, c}{c \mid ab} (\ldots) = \sum_\twoln{a, b, c, t}{ab = ct} (\ldots) &= \sum_s \sum_\twoln{a_1, t_1}{ (a_1, t_1) = 1 } \sum_\twoln{c, b}{ c t_1 = a_1 b } (\ldots) &\text{with \( a_1 := \frac as, t_1 := \frac ts \)} \\
    &= \sum_s \sum_\twoln{a_1, t_1}{ (a_1, t_1) = 1 } \sum_{a_1 \mid c} (\ldots) &\text{with \( b = \frac c{a_1} t_1 \)} \\
    &= \sum_s \sum_{a_1, t_1} \sum_{ r \mid (a_1, t_1) } \mu(r) \sum_{a_1 \mid c} (\ldots) \\
    &= \sum_{s, r} \sum_{a_2, t_2} \mu(r) \sum_{a_2 r \mid c} (\ldots) &\text{with \( a_2 := \frac{a_1}r, t_2 := \frac{t_1}r \)}.
\end{align*}
We set
\begin{align} \label{eqn: definition of F}
  F^\pm(c, m) := \frac{ar}c \int_0^\infty \! B^\pm\left( \frac{4 \pi}c \sqrt{m \xi} \right) f\left( \xi; ars, \frac{ct}a \right) \, d\xi
\end{align}
with
\begin{align*}
  B^+(\xi) &= Y_0(\xi), \quad & &B^-(\xi) = K_0(\xi), \quad & &\text{if \( \alpha = d \),} \\
  B^+(\xi) &= J_{\kappa - 1}(\xi), \quad & &B^-(\xi) = 0, \quad & &\text{if \( \alpha = a \),}
\end{align*}
and after renaming \(a_2\) and \(t_2\), we end up with
\begin{align*}
  R^\pm_{ABC} &:= \sum_{r, s, t} \frac{ \mu(r) }{r^2 st} \sum_a \sum_m \alpha(m) \sum_{ar \mid c} \frac{ S(h, \pm m; c) }c \frac{ F^\pm(c, m) }a \\
    &= \sum_{r, s, t} \frac{ \mu(r) }{r^2 st} \sum_a \frac{ R^\pm_{ABC}(a; r, s, t) }a
\end{align*}
where
\[ R^\pm_{ABC}(a; r, s, t) := \sum_m \alpha(m) \sum_{ar \mid c} \frac{ S(h, \pm m; c) }c F^\pm(c, m), \]
for which we need to find good bounds.
Note that the sums over \(a\) and \(c\) are supported in
\[ a \asymp \frac A{rs} \quad \text{and} \quad c \asymp \frac{AB}{rst}. \]
The function \( F^\pm(c, m) \) can be bound by
\[ F^\pm(c, m) \ll x^{1 + \varepsilon} \frac A{sc} \asymp x^{1 + \varepsilon} \frac{rt}B, \]
however, when \( m \gg \frac{ c^2 }x \) we can use \eqref{eqn: consequence of the recurrence relations} to get
\[ F^+(c, n) \ll \frac1{ x^{\frac\nu2 - \frac34} \Omega^{\nu - 1} } \frac{ c^{\nu - \frac12} }{ n^{\frac\nu2 + \frac14} } \frac As \quad \text{and} \quad F^-(c, n) \ll \frac1{ x^{\frac\nu2 - \frac34} } \frac{ c^{\nu - \frac12} }{ n^{\frac\nu2 + \frac14} } \frac As. \]
We set
\[ M_0^- := \frac{x^\varepsilon}x \left( \frac{AB}{rst} \right)^2 \quad \text{and} \quad M_0^+ := \frac{x^\varepsilon}{x \Omega^2} \left( \frac{AB}{rst} \right)^2, \]
and a standard exercise then shows that we can cut the sum over \(m\) in \( R_{ABC}^\pm \) at \( M_0^\pm \), so that it is sufficient to look at the sums
\begin{align}
  R^\pm_{ABC}(M) := \sum_{M < m \leq 2M} \alpha(m) \sum_{ar \mid c} \frac{ S(h, \pm m; c) }c F^\pm(c, m), \label{eqn: cut off sum}
\end{align}
where we have divided the range of summation over \(n\) into dyadic intervalls \( [M, 2M] \) with \( M = \frac{M_0^\pm}{2^k} \), where \(k\) runs over positive integers.

\subsection{Auxiliary estimates}

We want to use the Kuznetsov formula given in Theorem \ref{thm: Kuznetsov trace formula} for the inner sum in \eqref{eqn: cut off sum}.
To bring the functions \( F^\pm(c, n) \) into the right shape, we define
\[ \tilde F^\pm(c, m) := h(m) \frac{arc}{ 4\pi \sqrt{|h|m} } \int_0^\infty \! B^\pm\left( c \sqrt{ \frac\xi{|h|} } \right) f\left( \xi; ars, \frac{ 4\pi \sqrt{|h|m} }{c} \frac ta \right) \, d\xi, \]
where \( h(m) \) is a smooth and compactly supported bump function such that
\[ h(m) \equiv 1 \quad \text{for} \quad m \in [M, 2M], \quad \supp h \asymp M \quad \text{and} \quad h^{ (\nu) }(m) \ll \frac1{ M^\nu }. \]
Then we have
\[ F^\pm(c, m) = \tilde F^\pm \left( \frac{ 4\pi \sqrt{|h|m} }c, m \right) \quad \text{for} \quad m \in [M, 2M]. \]

In order to seperate the variable \(m\) we use Fourier inversion.
First define
\[ G_0(\lambda) := x^{1 + \varepsilon} \frac{rt}B \min\left( M, \frac1\lambda, \frac1{ M \lambda^2 } \right), \]
which is just a normalization factor.
We have
\[ \tilde F^\pm(c, m) = \int \! G_0(\lambda) G_\lambda^\pm(c) e(\lambda m) \, d\lambda, \quad G_\lambda^\pm(c) := \frac1{ G_0(\lambda) } \int \! \tilde F^\pm(c, m) e(-\lambda m) \, dm, \]
so that
\[ R_{ABC}^\pm(M) = \int \! G_0(\lambda) \sum_{M < m \leq 2M} \alpha(m) \e(\lambda m) \sum_{ar \mid c} \frac{ S(h, \pm m; c) }c G_\lambda^\pm\left( \frac{ 4\pi \sqrt{|h|m} }c \right) \, d\lambda. \]
Before going on, we need some good estimates for the Bessel transforms occuring in the Kuznetsov formula.
For convenience set
\[ W := \sqrt{ |h| M } \frac{rst}{AB} \quad \text{and} \quad Z := \sqrt{xM} \frac{rst}{AB}. \]

\begin{lemma} \label{lemma: bounds for the Bessel transforms of G}
  We have for \( M \ll M_0^- \),
  \begin{align}
    \hat G_\lambda^\pm(ic), \check G_\lambda^\pm(ic) &\ll W^{-2c} \quad & &\text{for} \quad 0 \leq c < \frac14, \label{eqn: First bound for Bessel transforms of G} \\
    \hat G_\lambda^\pm(c), \check G_\lambda^\pm(c), \tilde G_\lambda^\pm(c) &\ll \frac{ x^\varepsilon }{1 + c^\frac52} \quad & &\text{for} \quad c \geq 0. \label{eqn: Second bound for Bessel transforms of G}
  \end{align}
  If \( M_0^- \ll M \ll M_0^+ \), we have for any \( \nu \geq 0 \),
  \begin{align}
    \hat G_\lambda^\pm(ic), \check G_\lambda^\pm(ic) &\ll x^{-\nu} \quad & &\text{for} \quad 0 \leq c < \frac14, \label{eqn: Third bound for Bessel transforms of G} \\
    \hat G_\lambda^\pm(c), \check G_\lambda^\pm(c), \tilde G_\lambda^\pm(c) &\ll \frac{ x^\varepsilon }{ Z^\frac52 } \left( \frac Zc \right)^\nu \quad & &\text{for} \quad c \geq 0. \label{eqn: Fourth bound for Bessel transforms of G}
  \end{align}
\end{lemma}

\begin{proof}
  Since all occurring integrals can be interchanged, we can look directly at the Bessel transforms of \( \tilde F^\pm(c, m) \) and its first two partial derivatives in \(m\).
  We will confine ourselves with the treatment of  \( \tilde F^\pm(c, m) \), since the corresponding estimates for the derivatives can be shown the same way.
  
  First we want to use Lemma \ref{lemma: estimates for the Bessel transforms} to prove the first two bounds.
  Again we can look directly at the function inside the integral over \(\xi\), given by
  \[ H_1(c) := c B^\pm\left( c \sqrt{ \frac\xi{ |h| } } \right) f\left( \xi; ars, \frac{ 4\pi \sqrt{ |h| m } }c \frac ta \right), \]
  for which we have the bounds
  \[ \supp H_1(c) \asymp W \quad \text{and} \quad H_1^{ (\nu) }(c) \ll x^\varepsilon W \left( \frac{ x^\varepsilon }W \right)^\nu. \]
  Hence by the mentioned lemma
  \begin{align*}
    \hat H_1(ic), \check H_1(ic) &\ll W^{1 - 2c} \quad \text{for} \quad 0 \leq c < \frac14, \\
    \hat H_1(c), \check H_1(c), \tilde H_1(c) &\ll \frac{ x^\varepsilon W }{ 1 + c^\frac52 } \quad \text{for} \quad c \geq 0,
  \end{align*}
  from which we get \eqref{eqn: First bound for Bessel transforms of G} and \eqref{eqn: Second bound for Bessel transforms of G}.
  
  When \( M \gg M_0^- \), oscillation effects come into play.
  By using Lemma \ref{lemma: Bessel functions as oscillating functions} and partially integrating once over \(\xi\), we get
  \[ \tilde F^+(c, m) = -h(m) \frac{ar}{ 2\pi i \sqrt{ |h| m } } \Re\left( \int_0^\infty \! \e\left( \frac c{2\pi} \sqrt{ \frac\xi{ |h| } } \right) \tilde w(c) \, d\xi \right) \]
  with
  \[ \tilde w(c) := \frac\partial{\partial \xi} \left( \sqrt{ \xi |h| } v_Y\left( \frac c\pi \sqrt{ \frac\xi{ |h| } } \right) f\left( \xi; ars, \frac{ 4\pi t \sqrt{ |h| m } }{ac} \right) \right). \]
  It is hence enough to look at
  \[ H_2(c) := \e\left( \frac c{2\pi} \sqrt{ \frac\xi{ |h| } } \right) \tilde w(c), \]
  where we have the bounds
  \[ \supp \tilde w \asymp W, \quad \text{and} \quad \tilde w^{ (\nu) }(\xi) \ll \frac{ W^{1 - \nu} }{ Z^\frac32 } C(\xi) \quad \text{with} \quad C(\xi) := 1 + \left| w'\left( \frac{\xi - h}x \right) \right|. \]
  We use Lemma \ref{lemma: estimates for the Bessel transforms of oscillating functions for large alpha} with \( \alpha = \frac1{2\pi} \sqrt{ \frac\xi{ |h| } } \) and \( X = W \), which is possible since
  \[ W \ll x^{\omega + \varepsilon} \sqrt{ \frac{ |h| }x } \ll \frac1{ x^\varepsilon } \quad \text{and} \quad \alpha W \asymp Z \gg x^\varepsilon, \]
  and so we get
  \begin{align*}
    \hat H_2(ic), \check H_2(ic) &\ll x^{-\nu} \quad \text{for} \quad 0 \leq c < \frac14, \\
    \hat H_2(c), \check H_2(c), \tilde H_2(c) &\ll x^\varepsilon C(\xi) \frac W{ Z^\frac52 } \left( \frac Zc \right)^\nu \quad \text{for} \quad c \geq 0,
  \end{align*}
  which then give \eqref{eqn: Third bound for Bessel transforms of G} and \eqref{eqn: Fourth bound for Bessel transforms of G}.
\end{proof}

\subsection{Use of the Kuznetsov trace formula}

Now we are ready to apply the Kuznetsov trace formula.
We will only look at \( R_{ABC}^+(M) \) and we will assume that \( h \geq 1 \), since all other cases can be treated in very similar ways.
Here we use Theorem \ref{thm: Kuznetsov trace formula} on the inner sum,
\begin{align*}
  \sum_{ar \mid c} \frac{ S(h, m; c) }c &G_\lambda^+\left( \frac{ 4\pi \sqrt{hm} }c \right) = \sum_{j = 1}^\infty \frac{ \overline{\rho_j}(h) \rho_j(m) }{ \cosh(\pi \kappa_j) } \hat G_\lambda^+(\kappa_j) \\
    &+ \frac1\pi \sum_\mathfrak{c} \int_{-\infty}^\infty \! m^{ir} \overline{ \varphi_{ \mathfrak{c}, h } } \left( \frac12 + ir \right) \varphi_{ \mathfrak{c}, m } \left( \frac12 + ir \right) \hat G_\lambda^+(r) \, dr \\
    &+ \frac1{2\pi} \sum_\twoln{k \equiv 0 \smod 2}{ 1 \leq j \leq \theta_k(ar) } \frac{ i^k (k - 1)! }{ ( 4\pi \sqrt{m} )^{k - 1} } \overline{ \psi_{j, k} }(h) \psi_{j, k}(m) \tilde G_\lambda^+(k - 1).
\end{align*}
Hence we can write our sum as
\[ R_{ABC}^+(M) = \int \! G_0(\lambda) \left( \Xi_\text{exc.}(M) + \Xi_1(M) + \frac1\pi \Xi_2(M) + \frac1{2\pi} \Xi_3(M) \right) \, d\lambda, \]
where
\begin{align*}
  \Xi_\text{exc.}(M) &= \sum_{ \kappa_j \text{ exc.} } \hat G_\lambda^+(\kappa_j) \left( \frac{ \overline{\rho_j}(h) }{ \sqrt{ \cosh(\pi \kappa_j) } } \right) \Sigma_j^{( \text{exc.} )}(M), \\
  \Xi_1(M) &= \sum_{\kappa_j \geq 0} \hat G_\lambda^+(\kappa_j) \left( \frac{ \overline{\rho_j}(h) }{ \sqrt{ \cosh(\pi \kappa_j) } } \right) \Sigma_j^{(1)}(M), \\
  \Xi_2(M) &= \sum_\mathfrak{c} \int \! \hat G_\lambda^+(r) \left( \overline{ \varphi_{ \mathfrak{c}, h } }\left( \frac12 + ir \right) \right) \Sigma_{ \mathfrak{c},r  }^{(2)}(M) \, dr, \\
  \Xi_3(M) &= \sum_\twoln{k \equiv 0 \smod 2}{ 1 \leq j \leq \theta_k(ar) } \tilde G_\lambda^+(k - 1) \left( i^\frac k2 \sqrt{ \frac{ (k - 1)! }{ (4\pi)^{k - 1} } } \overline{ \psi_{j, k} }(h) \right) \Sigma_{j, k}^{(3)}(M),
  \intertext{and}
  \Sigma_j^{( \text{exc.} )}(M) := \Sigma_j^{(1)} (M) &:= \frac1{ \sqrt{ \cosh(\pi \kappa_j) } } \sum_{M < m \leq 2M} \alpha(m) \e(\lambda m) \rho_j(m), \\
  \Sigma_{ \mathfrak{c}, r }^{(2)} (M) &:= \sum_{M < m \leq 2M} \alpha(m) \e(\lambda m) m^{ir} \varphi_{ \mathfrak{c}, m } \left( \frac12 + ir \right), \\
  \Sigma_{j, k}^{(3)} (M) &:= i^\frac k2 \sqrt{ \frac{ (k - 1)! }{ (4\pi)^{k - 1} } } \sum_{M < m \leq 2M} \alpha(m) \e(\lambda m) m^{ -\frac{k - 1}2 } \psi_{j, k}(m).
\end{align*}

\( \Xi_\text{exc.}(M) \) needs a special treatment, which we will do in the following section.
First, we want to look at the other summands, and here we will restrict ourselves to \( \Xi_1 \), since the treatment of the other sums can be done along the same lines.

First assume \( M \ll M_0^- \).
We divide \( \Xi_1(M) \) into two parts:
\[ \Xi_1(M) = \sum_{\kappa_j \leq 1} (\ldots) + \sum_{1 < \kappa_j} (\ldots) =: \Xi_{ 1 \text{a} }(M) + \Xi_{ 1 \text{b} }(M). \]
For \( \Xi_{ 1 \text{a} }(M) \) we get using \eqref{eqn: Second bound for Bessel transforms of G}, Cauchy-Schwarz, Theorem \ref{thm: large sieve inequalities} and Lemma \ref{lemma: estimates for Fourier coefficients},
\begin{align*}
  \Xi_{ 1 \text{a} }(M) &\ll \max_{0 \leq \kappa_j \leq 1} \left| \hat G_\lambda^+(\kappa_j) \right| \sum_{\kappa_j \leq 1} \frac{ | \rho_j(h) | }{ \sqrt{ \cosh(\pi \kappa_j) } } \left| \Sigma_j^{(1)}(M) \right| \\
    &\ll x^\varepsilon \left( 1 + \frac{ (ar, h)^\frac12 h^\frac12 }{ar} \right)^\frac12 \left( 1 + \frac M{ar} \right)^\frac12 M^\frac12 \\
    &\ll x^\varepsilon \frac1{rt} \left( \frac{AB}{ x^\frac12 } + \frac{ A^\frac32 B^2 }x \right) \left( 1 + \frac{ (ar, h)^\frac14 h^\frac14 }{ A^\frac12 } \right)
\end{align*}
so that
\[ \int \! G_0(\lambda) \Xi_{ 1 \text{a} }(M) \, d\lambda \ll x^{\frac56 + \varepsilon} \left( 1 + \frac{ (ar, h)^\frac14 h^\frac14 }{ x^\frac16 } \right). \]
We split up the remainig sums into dyadic segments
\[ \Xi_1(M, K) := \sum_{K < \kappa_j \leq 2K} \hat G_\lambda^+(\kappa_j) \frac{ \overline{\rho_j}(h) }{ \sqrt{ \cosh(\pi \kappa_j) } } \Sigma_j^{(1)}(M), \]
and in the same way as above we get
\[ \Xi_1(M, K) \ll x^\varepsilon \frac1{rt} \left( \frac{AB}{ x^\frac12 } \frac1{ K^\frac12 } + \frac{ A^\frac32 B^2 }x \frac1{ K^\frac32 } \right) \left( 1 + \frac1K \frac{ (ar, h)^\frac14 h^\frac14 }{ A^\frac12 } \right), \]
which then gives
\[ \int \! G_0(\lambda) \Xi_{ 1 \text{b} }(M)  \, d\lambda \ll x^{\frac56 + \varepsilon} \left( 1 + \frac{ (ar, h)^\frac14 h^\frac14 }{ x^\frac16 } \right). \]

The case \( M \gg M_0^- \) is handled the same way:
We again divide \( \Xi_1(M) \) into two parts
\[ \Xi_1(M) = \sum_{\kappa_j \leq Z} (\ldots) + \sum_{Z < \kappa_j} (\ldots), \]
and this time we use the bound \eqref{eqn: Fourth bound for Bessel transforms of G}, which gives
\[ \int \! G_0(\lambda) \Xi_1(M) \, d\lambda \ll x^{\frac56 + \frac\omega2 + \varepsilon} \left( 1 + \frac{ (ar, h)^\frac14 h^\frac14 }{ x^\frac16 } \right). \]

The same bounds apply for \( \Xi_2(M) \) and \( \Xi_3(M) \), so that we end up with
\begin{align} \label{eqn: intermediate result for R(M)}
  R_{ABC}^+(M) = \int \! G_0(\lambda) \Xi_\text{exc.}(M) \, d\lambda + \BigO{ x^{\frac56 + \frac\omega2 + \varepsilon} \left( 1 + (ar, h)^\frac14 \frac{ h^\frac14 }{ x^\frac16 } \right) }.
\end{align}

\subsection{Treatment of the exceptional eigenvalues}

For \( M \gg M_0^- \), the exceptional eigenvalues pose no problem at all, since the Bessel transforms \( \hat G_\lambda^+(\kappa_j) \) are very small, as can be seen at \eqref{eqn: Third bound for Bessel transforms of G}.
So, \( \Xi_\text{exc.}(M) \) certainly does not exceed the size of the error term in \eqref{eqn: intermediate result for R(M)}.

For \( M \ll M_0^- \), this is a totally different story.
If we would bound \( \Xi_\text{exc.}(M) \) the same way as in the section above using \eqref{eqn: First bound for Bessel transforms of G}, we would end up with
\begin{align} \label{eqn: first bound for Xi_exc}
  \int \! G_0(\lambda) \Xi_\text{exc.}(M) \, d\lambda \ll x^{\frac56 + \theta + \varepsilon} \frac1{ h^\theta } \left( 1 + (ar, h)^\frac14 \frac{ h^\frac14 }{ x^\frac16 } \right).
\end{align}
With the currently best value for \(\theta\), this would weaken our result considerably.
However, we can reduce the effect of the exceptional eigenvalues by exploiting the fact that these eigenvalues appear infrequently.
Cauchy-Schwarz and \eqref{eqn: First bound for Bessel transforms of G} give
\begin{align*}
  \Xi_\text{exc.}(M) \ll \left( \sum_{ \kappa_j \text{ exc.} } \left( \frac1{ \sqrt{hM} } \frac{AB}{rst} \right)^{4i \kappa_j} | \rho_j(h) |^2 \right)^\frac12 \left( \sum_{ \kappa_j \text{ exc.} } \left| \Sigma_j^\text{exc.} (M) \right|^2 \right)^\frac12.
\end{align*}
The second factor be can treated with the large sieve inequalities.
Because of
\[ h^\frac12 \frac1{ \sqrt{hM} } \frac{AB}{rst} \gg x^{\frac12 - \varepsilon} \gg ar, \]
we can use Lemma \ref{lemma: lemma to treat the exceptional eigenvalues} to bound the first factor.
So,
\begin{align*}
  \Xi_\text{exc.}(M) &\ll x^\varepsilon \left( \frac1{ \sqrt{M} } \frac{AB}{rst} \frac1{ar} \right)^{2\theta} (ar, h)^\frac14 \left( 1 + \frac{ h^\frac12 }{ar} \right)^\frac12 \left( 1 + \frac M{ar} \right)^\frac12 M^\frac12 \\
    &\ll x^\varepsilon \frac1{rt} \frac{ x^\theta }{ A^{2\theta} } \left( \frac{AB}{ x^\frac12 } + \frac{A^\frac32 B^2}x \right) (ar, h)^\frac14 \left( 1 + \frac{ h^\frac14 }{ A^\frac12 } \right),
\end{align*}
and hence
\begin{align*}
  \int \! G_0(\lambda) \Xi_\text{exc.}(M) \, d\lambda &\ll x^\varepsilon \frac{ x^\theta }{ A^{2\theta} } \left( x^\frac12 A + A^\frac32 B \right) (ar, h)^\frac14 \left( 1 + \frac{ h^\frac14 }{ A^\frac12 } \right) \\
    &\ll x^{ \frac56 + \frac\theta3 + \varepsilon} (ar, h)^\frac14 \left( 1 + \frac{ h^\frac14 }{ x^\frac16 } \right),
\end{align*}
which is a substantial improvement to \eqref{eqn: first bound for Xi_exc}.

Eventually we get
\[ R_{ABC}^\pm(N) \ll x^{ \frac56 + \varepsilon} \left( x^\frac\omega2 + x^\frac\theta3 \right) (ar, h)^\frac14 \left( 1 + \frac{ h^\frac14 }{ x^\frac16 } \right), \]
which as a consequence gives the error term in \eqref{eqn: main sum}.

\subsection{The main term} \label{subsection: The main terms}

To finish the proof of \eqref{eqn: main sum}, we have to evaluate the main term, which occurs in the case \( \alpha(n) = d(n) \), and which is given by
\begin{align*}
  M_0(w) &= \sum_{a, b} \frac1{ab} \int \! \lambda_{h, ab}(\xi + h) w\left( \frac\xi x \right) h_{ABC} \left( a, b, \frac\xi{ab} \right) \, d\xi \\
    &= x \int \! w(\xi) \sum_{a, b} \frac{ \lambda_{h, ab}(x \xi) }{ab} h\left( a, b, \frac{x \xi}{ab} \right) \, d\xi + \BigO{ x^\varepsilon h },
\end{align*}
so effectively we are concerned with
\[ M_1(\xi) := \sum_{a, b} \frac{ \lambda_{h, ab}(x \xi) }{ab} H_1(a, b; \xi), \quad \text{where} \quad H_1(a, b; \xi) := h\left( a, b, \frac{x \xi}{ab} \right). \]

Using Mellin inversion this sum can be written as
\[ M_1(\xi) = \frac1{2\pi i} \sum_a \frac1a \int_{ (\sigma_1) } \! \hat H_1(a, s; \xi) \sum_{b = 1}^\infty \frac{ \lambda_{h, ab}(x \xi) }{ b^{1 + s} } \, ds, \quad \sigma_1 > 0, \]
where the Mellin transform of \( H_1(a, b; \xi) \) is given by
\[ \hat H_1(a, s; \xi) := \int_0^\infty \! H_1(a, b; \xi) b^{s - 1} \, db, \quad \Re(s) > 0. \]
A routine calculation then shows that for \( \Re(s) > 0 \),
\[ \sum_{b = 1}^\infty \frac{ \lambda_{h, ab}(x \xi) }{ b^{1 + s} } \, ds = \zeta(1 + s) \sum_{d = 1}^\infty \frac{ c_d(h) ( \log(x \xi) + 2\gamma - 2 \log d ) (a, d)^{1 + s} }{ d^{2 + s} }, \]
so that it is sufficient to look at
\begin{align}\label{eqn: first Mellin transformed main term}
  \tilde M_1(\xi, d) := \frac1{2\pi i} \sum_a \frac{ (a, d) }a \int_{ (\sigma_1) } \! \hat H_1(a, s; \xi) \zeta(1 + s) \frac{ (a, d)^s }{ d^s} \, ds.
\end{align}

Here we want to use the residue theorem.
\( \hat H_1(a, s; \xi) \) can be continued meromorphically to the whole complex plane with a simple pole at \( s = 0 \), and its Laurent series is given by
\begin{align*}
  \hat H_1(a, s; \xi) = 3 v_1(a) \frac1s &+ 3 v_1(a) \left( \log\frac{x \xi}a + C(a) \right) + \BigO{s},
\end{align*}
where
\[ C(a) := \int_0^\infty \! v_1'(b) \log b \, db + \frac13 \int_0^\infty \! v_1'(b) v_1\left( \frac{x \xi}{ab} \right) \log\frac{x \xi}{a b^2} \, db. \]
We also have that,
\[ \hat H_1(a, s; \xi) \ll \frac1{ |s| |s + 1| } x^{ \frac13 \Re(s) }. \]
Now we shift the line of integration in \eqref{eqn: first Mellin transformed main term} to \( \Re(s) = -1 + \varepsilon \) and the residue theorem gives
\[ \tilde M_1(\xi, d) = 3 M_{ 2 \text{a} }(\xi, d) + 3 M_{ 2 \text{b} }(\xi, d) + \BigO{ \frac{ d^{1 - \varepsilon} }{ x^{\frac13 - \varepsilon} }}, \]
where
\begin{align*}
  M_{ 2 \text{a} }(\xi, d) &:= \sum_a \frac{ (a, d) }a \log\frac{ (a, d) }a H_{ 2 \text{a} }(a), \quad M_{ 2 \text{b} }(\xi, d) := \sum_a \frac{ (a, d) }a H_{ 2 \text{b} }(a, \xi), \\
  \intertext{and}
  H_{ 2 \text{a} }(a) &:= v_1(a), \quad H_{ 2 \text{b} }(a; \xi) := v_1(a) \left( \log\frac{x \xi}d + \gamma + C_1(a) \right).
\end{align*}

The evaluation of these two sums can be done the same way as above using Mellin inversion and the residue theorem.
The appearing Dirichlet series can be continued meromorphically via
\begin{align*}
  \sum_a \frac{ (a, d) }{ a^{1 + s} } \log\frac{ (a, d) }a &= \sum_{r \mid d} \frac{ \mu(r) }r \sigma_s\left( \frac dr \right) \frac{ ( \zeta'(1 + s) - \zeta(1 + s) \log r ) }{d^s}, \\
  \sum_a \frac{ (a, d) }{ a^{1 + s} } &= \frac{ \zeta(1 + s) }{d^s} \sum_{r \mid d} \frac{ \mu(r) }r \sigma_s\left( \frac dr \right),
\end{align*}
which are identites for \( \Re(s) > 0 \).
Furthermore, the Mellin transforms \( \hat H_{ 2 \text{a} }(s) \) and \( \hat H_{ 2 \text{b} }(s; \xi) \) too have a meromorphic continuation to the whole complex plane, both with a simple pole at \( s = 0 \), and with Laurent series of the form
\begin{align*}
  \hat H_{ 2 \text{a} }(s) &= \frac1s + P_{ 1 \text{a} }(\log x) + s P_{ 2 \text{a} }(\log x) + \BigO{s^2}, \\
  \hat H_{ 2 \text{b} }(s) &= \frac1s P_{ 1 \text{b} }(\log x, \log \xi) + P_{ 2 \text{b} }(\log x, \log \xi) + \BigO{s},
\end{align*}
where \( P_{ 1 \text{a} } \) and \( P_{ 1 \text{b} } \) are linear polynomials, and \( P_{ 2 \text{a} } \) and \( P_{ 2 \text{b} } \) quadratic ones (which may depend on \(d\) and \(v\)).
We also have the bounds
\[ \hat H_{ 2 \text{a} }(s), \hat H_{ 2 \text{b} }(s; \xi) \ll \frac1{ |s| |s + 1| } x^{ \frac13 \Re(s) + \varepsilon }. \]
Now applying the residue theorem the same way as before we get
\[ \tilde M_1(\xi, d) = P(\log x, \log\xi) + \BigO{ \frac{ d^{1 - \varepsilon} }{ x^{\frac13 - \varepsilon} } }, \]
where \(P\) is a quadratic polynomial depending only on \(d\), which as a consequence then gives \eqref{eqn: main sum}.

\section{Proof of Theorems \ref{thm: main theorem for the dual sum of d(n)} and \ref{thm: main theorem for the dual sum of a(n)}}

Now we are interested in the sums
\[ \sum_{n = 1}^{N - 1} d_3(n) d(N - n) \quad \text{and} \quad \sum_{n = 1}^{N - 1} d_3(n) a(N - n), \]
and as before we can consider both sums simultaneously, so that we will stick to the convention that \( \alpha(n) \) is a placeholder for \( d(n) \) or \( a(n) \).
We first construct a smooth decomposition of the unit interval in a form suiting our needs.
There exist smooth and compactly supported functions \( \tilde u_i : \mathbb{R} \rightarrow [0, \infty) \), \( i \geq 1 \), such that
\[ \supp \tilde u_i \subset \left[ \frac1{ 2^{i + 2} }, \frac1{ 2^i } \right], \quad \text{and} \quad \sum_{i = 1}^\infty \tilde u_i(\xi) = 1 \quad \text{for} \quad \xi \in \left(0, 1/4 \right]. \]
For \( i \geq 1 \) we then define
\[ u_i(\xi) := \tilde u_i\left( \frac\xi{N - 1} \right), \quad u_{-i}(\xi) := u_i\left(N - 1 - \xi \right) \quad \text{and} \quad u_0(\xi) := 1 - u_1(\xi) - u_{-1}(\xi), \]
so that by construction
\[ \sum_{ i \in \mathbb{Z} } u_i(\xi) = 1 \quad \text{for} \quad \xi \in (0, N - 1). \]

We have
\[ \sum_{n = 1}^{N - 1} d_3(n) \alpha(N - n) = \sum_{ i \in \mathbb{Z} } \sum_n u_i(n) d_3(n) \alpha(N - n), \]
hence it is enough to look at the sums
\[ \Psi_i(N) := \sum_n u_i(n) d_3(n) \alpha(N - n). \]
The evaluation of these sums follows the same path as in section \ref{section: main proof}, we will therefore use in large parts the same notation and omit many details.

For the sake of easier notation, we will leave out the \(i\)-subscript from now on.
So \( u(\xi) := u_i(\xi) \), and we have
\[ \supp u(\xi) \subseteq \left[ \frac x2, 2x \right] \quad \text{for} \quad i \geq 0, \quad \supp u(\xi) \subseteq \left[ N - 2x, N - \frac x2 \right] \quad \text{for} \quad i < 0, \]
with
\[ x := \frac N{ 2^{|i| + 1} }. \]
A first trivial bound is then given by
\[ \Psi(N) := \Psi_i(N) \ll N^\varepsilon x. \]
The decomposition we use for \( d_3(n) \) is the same as in \eqref{eqn: decomposition of d_3(n)}, but with a different normalization, namely
\[ v_1(\xi) := v\left( \frac\xi{ (N - 1)^\frac13} \right), \quad v_2(\xi) := v\left( \frac\xi{ \sqrt{ \frac{N - 1}a } } \right). \]

It is enough to look at
\[ \Psi_{ABC} := \sum_{a, b, c} h_{ABC}(a, b, c) \alpha(N - abc) u(abc) = \sum_{a, b} \sum_{ m \equiv N (ab) } \alpha(m) f(m; a, b) \]
with
\[ f(m; a, b) := h_{ABC}\left( a, b, \frac{N - m}{ab} \right) u(N - m). \]
After using the Voronoi formula and reordering the sums, we get as a possible main term
\[ M_0(N) := \sum_{a, b} \frac1{ab} \int \! \lambda_{N, ab}(\xi) f(\xi; a, b) \, d\xi, \]
and as error terms we eventually have to deal with
\[ R_{ABC}^\pm(M) := \sum_{M < m \leq 2M} \alpha(m) \sum_{ar \mid c} \frac{ S(N, \pm m; c) }c F^\pm(c, m), \]
where \( F^\pm(c, m) \) is defined the same way as in \eqref{eqn: definition of F}.
We only need to look at the \( R_{ABC}^\pm(M) \) with \( M \leq M_0^\pm \), given by
\begin{align*}
  &M_0^+ := \frac{ N^{1 + \varepsilon} }{x^2} \left( \frac{AB}{rst} \right)^2, \quad M_0^- := \frac{ N^\varepsilon }N \left( \frac{AB}{rst} \right)^2 \quad \text{for} \quad i \geq 0, \\
  \intertext{and}
  &M_0^+ := M_0^- := \frac{ N^\varepsilon }x \left( \frac{AB}{rst} \right)^2 \quad \text{for} \quad i < 0,
\end{align*}
since otherwise \( R_{ABC}^\pm(M) \) is small.

We bring again everything into the right shape for the use of the Kuznetsov formula by setting
\[ \tilde F^\pm(c, m) := F^\pm\left( \frac{ 4\pi \sqrt{Nm} }c, m \right) \]
and using Poisson inversion so separate the variable \(m\), so that
\[ R_{ABC}^\pm(M) = \int \! G_0(\lambda) \sum_{M < m \leq 2M} \alpha(m) \e(\lambda m) \sum_{ar \mid c} \frac{ S(N, \pm m; c) }c G_\lambda^\pm\left( \frac{ 4\pi \sqrt{Nm} }c \right) \, d\lambda, \]
where
\[ G_0(\lambda) := N^\varepsilon x \frac{rt}B \min\left( M, \frac1\lambda, \frac1{ M \lambda^2 } \right). \]
Set
\[ W := \sqrt{NM} \frac{rst}{AB}. \]
When bounding the Bessel transforms, we have to distinguish between the cases \( i \geq 0 \) and \( i < 0 \).

\subsection{The case \( i \geq 0 \)}

In this case, we have the following bounds when \( M \ll M_0^- \),
\begin{align*}
  \hat G_\lambda^\pm(ic), \check G_\lambda^\pm(ic) &\ll N^\varepsilon W^{-2c} \quad & &\text{for} \quad 0 \leq c < \frac14, \\
  \hat G_\lambda^\pm(c), \check G_\lambda^\pm(c), \tilde G_\lambda^\pm(c) &\ll \frac{ N^\varepsilon }{1 + c^\frac52} \quad & &\text{for} \quad c \geq 0,
\end{align*}
while for \( M_0^- \ll M \ll M_0^+ \) we have
\begin{align*}
  \hat G_\lambda^\pm(ic), \check G_\lambda^\pm(ic) &\ll \frac{ N^\varepsilon }W \quad & &\text{for} \quad 0 \leq c < \frac14, \\
  \hat G_\lambda^\pm(c), \tilde G_\lambda^\pm(c) &\ll \frac{ N^\varepsilon }W \left( \frac{ W^\frac12 }c \right)^\nu \quad & &\text{for} \quad c \geq 0, \\
  \check G_\lambda^\pm(c)  &\ll \frac{ N^\varepsilon }{ W^\frac32 } \left( \frac Wc \right)^\nu \quad & &\text{for} \quad c \geq 0.
\end{align*}
All these bounds can be derived the same way as in Lemma \ref{lemma: bounds for the Bessel transforms of G}.
There are two slight differences, though:
Applying partial integration once over \(\xi\) is useless here.
Furthermore, instead of Lemma \ref{lemma: estimates for the Bessel transforms of oscillating functions for large alpha} we need to use Lemma \ref{lemma: estimates for the Bessel transforms of oscillating functions for alpha near 1}.

Now applying the Kuznetsov formula and the large sieve inequalities, we get that
\[ R_{ABC}^+(M) \ll N^{ \frac{11}{12} + \varepsilon } \quad \text{and} \quad R_{ABC}^-(M) \ll N^{ \frac{11}{12} + \varepsilon } + \frac{ N^{\frac43 + \varepsilon} }{ x^\frac12}. \]
In contrast to section \ref{section: main proof}, the exceptional eigenvalues cause no problem at all.

\subsection{The case \( i < 0 \)}

The bounds for the Bessel transforms for \( M \ll \frac{ N^\varepsilon}N \left( \frac{AB}{rst} \right)^2 \) are given by
\begin{align*}
  \hat G_\lambda^\pm(ic), \check G_\lambda^\pm(ic) &\ll N^\varepsilon W^{-2c} \quad & &\text{for} \quad 0 \leq c < \frac14, \\
  \hat G_\lambda^\pm(c), \check G_\lambda^\pm(c), \tilde G_\lambda^\pm(c) &\ll \frac{ N^\varepsilon }{1 + c^\frac52} \quad & &\text{for} \quad c \geq 0,
\end{align*}
and for \( \frac{ N^\varepsilon}N \left( \frac{AB}{rst} \right)^2 \ll M \ll M_0^\pm \) by
\begin{align*}
  \hat G_\lambda^\pm(ic), \check G_\lambda^\pm(ic) &\ll \frac{ N^\varepsilon }W \quad & &\text{for} \quad 0 \leq c < \frac14, \\
  \hat G_\lambda^\pm(c), \check G_\lambda^\pm(c), \tilde G_\lambda^\pm(c) &\ll \frac{ N^\varepsilon }W \quad & &\text{for} \quad c \geq 0, \\
  \hat G_\lambda^\pm(c), \check G_\lambda^\pm(c), \tilde G_\lambda^\pm(c) &\ll \frac{ N^\varepsilon }{ c^\frac52 } \left( 1 + \frac W{ c^\frac12 } \right) \quad & &\text{for} \quad c \gg W.
\end{align*}
Another use of the Kuznetsov formula gives
\[ R_{ABC}^\pm(M) \ll N^{ \frac{11}{12} + \varepsilon }. \]
So, altogether we have for all \( i \in \mathbb{Z} \),
\[ R_{ABC}^\pm(M) \ll N^{ \frac{11}{12} + \varepsilon } + \frac{ N^{\frac43 + \varepsilon} }{ x^\frac12}. \]
We use this bound for \( x \gg N^\frac89 \) and otherwise bound trivially, to get the error terms claimed in Theorems \ref{thm: main theorem for the dual sum of d(n)} and \ref{thm: main theorem for the dual sum of a(n)}.

\subsection{The main term}

To finish the proof, we have to evaluate the main term, which occurs in the case \( \alpha = d \) and which is given by
\begin{align*}
  M_0(N) &= \sum_{a, b} \frac1{ab} \int_1^{N - 1} \! \lambda_{N, ab}(\xi) h\left(a, b, \frac{N - \xi}{ab} \right) \, d\xi \\
    &= N \int_0^1 \! \sum_{a, b} \frac{ \lambda_{N, ab}( N(1 - \xi) ) }{ab} h\left( a, b, \frac{ (N - 1)\xi }{ab} \right) \, d\xi + \BigO{ N^\varepsilon }.
\end{align*}
This, too, can be done the same way as in section \ref{subsection: The main terms}, so we will just state some intermediate results.
It is enough to look at
\[ M_1(\xi) = \sum_{a, b} \frac{ \lambda_{N, ab}( N(1 - \xi) ) }{ab} h\left( a, b, \frac{ (N - 1)\xi }{ab} \right) \, d\xi, \]
and this sum can be evaluated by using Mellin inversion and the residue theorem, so that we get
\begin{align*}
  M_1(\xi) &= 3 \sum_{d = 1}^\infty \frac{ c_d(N) \left( \log( N (1 - \xi) ) + 2\gamma - 2 \log d \right) }{d^2} \left( M_{ 2 \text{a} }(\xi, d) + M_{ 2 \text{b} }(\xi, d) \right) \\
    &\phantom{ ={} } + \BigO{ \frac1{ N^{\frac13 - \varepsilon} \xi^{1 - \varepsilon} } },
\end{align*}
with
\begin{align*}
  M_{ 2 \text{a} }(d) &= \sum_a \frac{ (a, d) }a \log\left( \frac{ (a, d) }a \right) v_1(a), \\
  M_{ 2 \text{b} }(\xi, d) &= \sum_a \frac{ (a, d) }a v_1(a) \left( \log\left( \frac{ N\xi }d \right) + \gamma + C(a) \right),
\end{align*}
and
\[ C(a) = \int_0^\infty \! v_1'(b) \log b \, db + \frac13 \int_0^\infty \! v_1'(b) v_1\left( \frac{ (N - 1)\xi }{ab} \right) \log\left( \frac{ (N - 1)\xi }{ ab^2 } \right) \, db . \]

The evaluation of \( M_{ 2 \text{a} }(d) \) and \( M_{ 2 \text{b} }(\xi, d) \) follows the usual pattern, and as result we get
\[ M_{ 2 \text{a} }(d) + M_{ 2 \text{b} }(\xi, d) = \sum_{r \mid d} \frac{ \mu(r) }r \sum_{ m \mid \frac dr } P_2(\log N, \log d, \log r, \log m) + \BigO{ \frac{ d^{1 - \varepsilon} }{ N^{\frac13 - \varepsilon} \xi^{1 - \varepsilon} } }, \]
where \( P_2 \) is a quadratic polynomial (which depends on \(\xi\)).
From this we see that \( M_0(N) \) has the form
\[ M_0(N) = N \sum_{d = 1}^\infty \frac{ c_d(N) }{d^2} \sum_{r \mid d} \frac{ \mu(r) }r \sum_{ m \mid \frac dr } P_3(\log N, \log d, \log r, \log m) + \BigO{ N^{\frac23 + \varepsilon} }, \]
with a cubic polynomial \( P_3 \).

We want to reshape this result a little bit.
Set
\[ G(\alpha, \beta, \gamma, \delta) := N^\alpha \sum_{d = 1}^\infty \frac{ c_d(N) }{ d^{2 - \beta} } \sum_{r \mid d} \frac{ \mu(r) }{ r^{1 - \gamma} } \sigma_\delta\left( \frac dr \right), \]
so that the main term can be stated in terms of the partial derivatives of \(G\) up to third order evaluated at \( (0, 0, 0, 0) \).
A lengthy but elementary calculation shows that
\begin{align*}
  G(\alpha, \beta, \gamma, \delta) &= N^\alpha \sum_{d \mid N} \sum_\twoln{c \mid d}{b \mid c} \mu\left( \frac dc \right) \frac{ c^{1 - \gamma + \delta} }{ d^{2 - \gamma - \beta} b^\delta } \sum_{ (r, d) = 1 } \sum_\twoln{ (s, br) = 1 }{ (d, rs) = 1 } \frac{ \mu^2(r) \mu(s) \mu(d) }{ r^{3 - \gamma - \delta} s^{2 - \beta - \delta } d^{2 - \beta} } \\
    &= C(\beta, \gamma, \delta) N^\alpha \sum_{d \mid N} \frac{ \chi_1(d) }{ d^{1 - \beta} } \sum_{c \mid d} \chi_2\left( \frac dc \right) \chi_3(c),
\end{align*}
with
\[ C(\beta, \gamma, \delta) := \frac1{ \zeta(2 - \delta) } \prod_p \left( 1 - \frac{ p^{1 - \gamma + \delta} - 1 }{ p^{1 - \gamma} ( p^{2 - \beta} - 1 ) } \right) \]
and \( \chi_1 \), \( \chi_2 \), and \( \chi_3 \) defined as in \eqref{eqn: definition of chi_1, chi_2 and chi_3}.
This eventually gives Theorem \ref{thm: main theorem for the dual sum of d(n)}.

\bibliography{ConvSumsArt}

\end{document}